\documentclass[12pt]{article}
\usepackage{amsmath,amssymb,amsthm,amsfonts,mathrsfs}
\usepackage{appendix}
\usepackage{array,enumitem,graphics,xcolor,yfonts,colonequals}
\usepackage{stmaryrd}
\usepackage[alphabetic]{amsrefs}
\usepackage[all,cmtip]{xy}
\usepackage{tikz-cd}
\usepackage[colorlinks,anchorcolor=blue,citecolor=blue,linkcolor=blue,urlcolor=blue,bookmarksopen=true]{hyperref}
\usepackage{comment}
\usepackage{mathtools}
\usepackage[margin=1in]{geometry}

\usepackage{titlesec}
\titleformat{\section}
  {\normalfont\large\bf}{\thesection}{1em}{}
\titleformat{\subsection}[runin]
  {\normalfont\normalsize\bf}{\thesubsection}{1em}{}

\usepackage{fancyhdr}
\fancypagestyle{titlepage}{}

\AtBeginDocument{%
	\def\MR#1{}
}

\newcommand{\bC}{\mathbb{C}}    
\newcommand{\bH}{\mathbb{H}}    
\newcommand{\bP}{\mathbb{P}}    
\newcommand{\bQ}{\mathbb{Q}}    
\newcommand{\bR}{\mathbb{R}}    
\newcommand{\bZ}{\mathbb{Z}}    

\newcommand{\cD}{\mathcal{D}}
\newcommand{\cO}{\mathcal{O}}   

\newcommand{\Coh}{\mathrm{Coh}} 
\newcommand{\Db}{\mathrm{D^b}}  
\newcommand{\Geo}{\mathrm{Geo}}
\newcommand{\GL}{\mathrm{GL}}           
\newcommand{\hyp}{\mathrm{hyp}}
\newcommand{\mass}{\mathrm{mass}}
\newcommand{\phase}{\mathrm{phase}}
   
\newcommand{\PSL}{\mathrm{PSL}}         
\newcommand{\rank}{\mathrm{rank}}

\newcommand{\Stab}{\mathrm{Stab}} 

\newcommand{\Hom}{\operatorname{Hom}}   

\newtheorem*{thm*}{Theorem}
\newtheorem*{prop*}{Proposition}
\newtheorem*{cor*}{Corollary}
\newtheorem*{ques*}{Question}
\newtheorem{thm}{Theorem}[section]
\newtheorem{prop}[thm]{Proposition}

\newtheorem{lemma}[thm]{Lemma}

\numberwithin{equation}{section}

\theoremstyle{definition}
\newtheorem{defn}[thm]{Definition}
\newtheorem{eg}[thm]{Example}

\newtheorem{rmk}[thm]{Remark}
\newtheorem{notn}[thm]{Notation}

\begin{document}
\title{A space of stability conditions that \\ is not a length space}
\author{Yu-Wei Fan}
\date{}
\maketitle

\begin{abstract}
We prove that the space of Bridgeland stability conditions, when equipped with the canonical metric, is not a length space in general. This resolves a question posed by Kikuta in the negative. Furthermore, we introduce two modified metrics on the stability spaces, which may exhibit better metric properties.
\end{abstract}

\section{Introduction}
\label{sec:Intro}

Stability conditions on triangulated categories were first introduced by Bridgeland in his seminal work \cite{BriStab}, motivated by ideas from string theory. Since then, the theory of stability conditions has found profound applications across diverse areas of mathematics, including algebraic geometry, mirror symmetry, symplectic geometry, and representation theory.

The space of stability conditions $\Stab(\cD)$ admits a canonical metric $d$. A fundamental result proved in \cite{BriStab} states that, with the topology on $\Stab(\cD)$ induced by the metric $d$, the natural map 
$$
\Stab(\cD) \rightarrow \Hom(\Lambda, \mathbb{C}); \qquad \sigma = (Z, P) \mapsto Z 
$$ 
is a local homeomorphism. In particular, this endows $\Stab(\cD)$ with the structure of a complex manifold.

The metric properties of $\Stab(\cD)$ have been extensively investigated from various perspectives:
\begin{itemize}
    \item \emph{Completeness}: The metric space $(\Stab(\cD), d)$ was shown to be complete by Woolf \cite{Woolf}*{Theorem~3.6}.
    \item \emph{Non-positive curvature}: Connections to CAT(0) geometry have been conjectured and explored in \cite{BayerBridgeland}*{Remark~1.5}, \cite{Allcock}*{Section~7}, \cite{Smith}*{Section~3.2}, and \cite{Kikuta}.
    \item \emph{Weil–-Petersson-type metrics}: A Weil–-Petersson metric on stability spaces was proposed in \cite{FKY}.
\end{itemize}
The metric properties are also closely related to recent developments on (partial) compactifications of stability spaces, as seen in works such as \cites{BDL,BPPW,Bolognese,DHL,BMS,KKO}.

It is proven in \cite{Kikuta} that $(\Stab(\cD),d)$ and $(\Stab(\cD)/\bC,\bar{d})$ are, in general, not CAT(0) spaces. A natural follow-up question arises concerning a weaker geometric property, is proposed in the same paper:
\begin{ques*}[\cite{Kikuta}*{Question~5.3}]
Is the metric space $(\Stab(\cD),d)$ a length space?
\end{ques*}

We answer this question negatively by establishing the following result:

\begin{thm}
\label{thm:main}
Let $\Db(\bP^1)$ denote the bounded derived category of coherent sheaves on the projective line. Then neither the space of stability conditions $(\Stab(\Db(\bP^1)),d)$ nor its quotient $(\Stab(\Db(\bP^1))/\bC,\bar{d})$ is a length space.
\end{thm}

The proof of this theorem is given in Section~\ref{subsec:proof-of-thm}. Let us outline the main idea of the proof. 
The space $\Stab(\Db(\bP^1))$ is completely described by Okada \cite{Okada}: It consists of a \emph{geometric} chamber $\Stab^\Geo(\Db(\bP^1))$, with $\bZ$-many \emph{algebraic} chambers attached along its boundary. The key steps are:
\begin{enumerate}[label=(\roman*)]
    \item Choose stability conditions $\sigma_1$ and $\sigma_2$, where $\sigma_1$ is \emph{geometric} and lies far from the boundary $\partial\overline{\Stab^\Geo(\Db(\bP^1))}$, while $\sigma_2\in\partial\overline{\Stab^\Geo(\Db(\bP^1))}$ achieves the infimum 
$$
d(\sigma_1,\sigma_2)=\inf\left\{d(\sigma_1,\tilde{\sigma}):\tilde{\sigma}\in\partial\overline{\Stab^\Geo(\Db(\bP^1))}\right\}.
$$
    \item\label{item:dist-preserve} Choose a small deformation $\sigma_3$ of $\sigma_2$ where:
    \begin{itemize}
        \item $\sigma_3$ lies in an algebraic chamber and has positive distance to the boundary 
        $$
        \inf\left\{d(\sigma_3,\tilde{\sigma}):\tilde{\sigma}\in\partial\overline{\Stab^\Geo(\Db(\bP^1))}\right\}=\epsilon>0,
        $$
        \item \emph{distance-preserving}: $d(\sigma_1,\sigma_2)=d(\sigma_1,\sigma_3)$.
    \end{itemize}
    \item Then, the distance between $\sigma_1$ and $\sigma_3$ cannot be realized by any path connecting them. Indeed, any such path must cross the boundary $\partial\overline{\Stab^\Geo(\Db(\mathbb{P}^1))}$, but $\sigma_2$ already minimizes the distance from $\sigma_1$ to the boundary. Consequently, every path joining $\sigma_1$ and $\sigma_3$ has length at least $d(\sigma_1,\sigma_3)+\epsilon$.
\end{enumerate}
This proves that $(\Stab(\Db(\mathbb{P}^1)), d)$ is not a length space.

\begin{rmk}
\label{rmk:algebraic-not-good}
The proof hinges on the ability to deform $\sigma_2$ into the algebraic chamber while preserving its distance to $\sigma_1$ (the distance-preserving property in Step~\ref{item:dist-preserve}).
This phenomenon reflects the distinct behavior of the distance function $d$ in geometric versus algebraic chambers.
These observations suggest that when $\Db(X)$ admits only geometric stability conditions, the resulting metric space may possess more favorable properties.

A concrete example occurs when $X$ is a curve of genus $g(X) \geq 1$: In this case, there is an isometry $(\Stab(\Db(X))/\mathbb{C}, \bar{d})\cong(\mathbb{H}, \frac{1}{2}d_{\text{hyp}})$, where $d_{\text{hyp}}$ denotes the hyperbolic metric on the upper half-plane \cite{Woolf}*{Proposition 4.1}.
It would be interesting to investigate whether similar metric properties hold for higher-dimensional varieties whose derived categories admit only geometric stability conditions, for instance, when the Albanese morphism is finite \cite{FLZ}.
\end{rmk}

Motivated by Theorem~\ref{thm:main}, we investigate alternative metrics on stability spaces that may exhibit more favorable geometric properties.
We introduce and analyze two modified metrics, focusing particularly on their behavior for $\bP^1$.
The first metric arises as a natural degeneration of $(\Stab(\mathcal{D}), d)$. Recall that the original distance function (see Notation~\ref{notn:distance-mass-phase}) is defined as the maximum of the mass and phase components:
$$
d(\sigma_1,\sigma_2)=\max\left\{d_\mass(\sigma_1,\sigma_2),d_\phase(\sigma_1,\sigma_2)\right\}.
$$
By considering the limit where the phase component goes to zero, we obtain a metric space:
$$
\Stab_\mass(\cD)\coloneqq\Stab(\cD)/\sim_m, \quad \text{ where } \sigma_1\sim_m\sigma_2 \text{ if } d_\mass(\sigma_1,\sigma_2)=0,
$$
equipping with the metric $d_\mass$.
This type of metric degeneration also appears in \cite{FFHKL}*{Section~2.1} and \cite{BaeLee}*{Remark~10.13}.
Note that the free $\mathbb{C}\cong\mathbb{R}^2$-action on $\Stab(\mathcal{D})$ descends to a free $\mathbb{R}$-action on $\Stab_\mass(\mathcal{D})$, since one of the $\bR$-components acts trivially on mass distances (see Remark~\ref{rmk:R-action-on-Mass}).

In the case of $\Stab(\Db(\mathbb{P}^1))$, we observe that the degeneration (identifying $\sigma_1$ and $\sigma_2$ when $d_\mass(\sigma_1,\sigma_2)=0$) collapses all algebraic chambers onto the boundary of the geometric chamber. This yields an isomorphism
$$
\Stab_\mass(\Db(\mathbb{P}^1))/\mathbb{R}\cong\overline{\Stab^\Geo_\mass(\Db(\bP^1))}/\bR.
$$
As suggested in Remark~\ref{rmk:algebraic-not-good}, the absence of algebraic stability conditions leads us to investigate whether the induced metric $\bar{d}_\mass$ on $\Stab_\mass(\Db(\mathbb{P}^1))/\mathbb{R}$ enjoys better properties. We address this question through the following result.

\begin{thm}
There is an isometry
$$
\left(\Stab_\mass(\Db(\mathbb{P}^1))/\mathbb{R},\bar{d}_\mass\right)\cong\left(\bH\cup(\bR\backslash\bZ),\frac{1}{2}d_\bZ\right),
$$
where $d_\bZ$ is a modified metric on $\bH$ that extends to $\bR\backslash\bZ$ (see Definition~\ref{defn:H-metric-Z}). 
Moreover:
\begin{itemize}
    \item The space is geodesic (and therefore a length space).
    \item The space is not uniquely geodesic (and therefore not CAT(0)).
\end{itemize}
\end{thm}

The second modified metric we consider is the \emph{length metric} construction: Given any metric space $(X,d)$, there exists an associated length metric $d_\ell$ for which $(X,d_\ell)$ is always a length space (see Definition~\ref{defn:length-metric}). While the topology induced by $d_\ell$ is always at least as fine as that of $d$, these topologies may differ in general.
This leads naturally to the question of whether the canonical metric $d$ and its length metric $d_\ell$ induce the same topology on stability spaces.

\begin{prop}
Assume that $(\overline{\Stab^\Geo(\Db(\bP^1))}/\bC,\bar{d})$ is a length space. Then the topology of $(\Stab(\Db(\bP^1))/\bC,\bar{d})$ coincides with the one induced by its associated length metric $\bar{d}_\ell$.
\end{prop}

\noindent\textbf{Acknowledgment.}
The author thanks Yi~Huang for stimulating discussions about hyperbolic metrics, and is grateful to Guillaume~Tahar for his invaluable suggestions that improved this article.

\section{Preliminaries}

In this section, we recall the background on stability conditions on triangulated categories in Section~\ref{subsec:Preliminary-Stab}, review basic definitions in metric geometry in Section~\ref{subsec:metric-geometry}, and introduce the modified metric $d_\bZ$ on the upper half-plane in Section~\ref{subsec:metric-on-H}.

\subsection{Stability conditions on triangulated categories.}
\label{subsec:Preliminary-Stab}
We recall the notion of stability conditions on triangulated categories, and the canonical metric on the space of stability conditions, introduced by Bridgeland \cite{BriStab}.

Let $\Lambda$ be a free abelian group of finite rank, $v\colon K_0(\cD)\rightarrow\Lambda$ be a group homomorphism, and $||\cdot||$ be a norm on $\Lambda\otimes_\bZ\bR$.

\begin{defn}[\cite{BriStab}*{Definition~5.1}]
A stability condition $\sigma=(Z,P)$ on a triangulated category $\cD$ (with respect to $\Lambda$ and $v$) consists of a group homomorphism $Z\colon\Lambda\rightarrow\bC$ and a collection $P=\{P(\phi)\}_{\phi\in\bR}$ of full additive subcategories of $\cD$, such that:
\begin{enumerate}[label=(\roman*)]
    \item for $E\in P(\phi)\backslash\{0\}$, $Z(v(E))\in\bR_{>0}\cdot e^{i\pi\phi}$,
    \item for all $\phi\in\bR$, $P(\phi+1)=P(\phi)[1]$,
    \item if $\phi_1>\phi_2$ and $E_i\in P(\phi_i)$, then $\Hom(E_1,E_2)=0$,
    \item for each nonzero object $E\in\cD$, there is a collection of exact triangles
\begin{equation}
\label{eqn:HNfil}
\xymatrix{
0=E_0 \ar[r] & E_1\ar[d]  \ar[r] & E_2 \ar[r] \ar[d]& \cdots \ar[r]& E_{k-1}\ar[r] & E \ar[d]\\
                    & A_1 \ar@{-->}[lu]& A_2 \ar@{-->}[lu] &  & & A_k  \ar@{-->}[lu] 
}    
\end{equation}
    where $A_i\in P(\phi)$ and $\phi_1>\cdots>\phi_k$,
    \item there exists a constant $C>0$ such that $|Z(v(E))|\geq C\cdot||v(E)||$ holds for any $E\in \cup_\phi P(\phi)\backslash\{0\}$.
\end{enumerate}
The set of such stability conditions is denoted by $\Stab_\Lambda(\cD)$.
\end{defn}

\begin{notn}
With the Harder--Narasimhan filtration (\ref{eqn:HNfil}), define
$$
m_\sigma(E)\coloneqq\sum_{i=1}^k|Z(A_i)|; \quad \phi_\sigma^+(E)\coloneqq\phi_1, \quad \phi_\sigma^-(E)\coloneqq\phi_k.
$$
\end{notn}

\begin{defn}[\cite{BriStab}*{Proposition~8.1}]
\label{defn-bridgeland-dist}
Let $\cD$ be a triangulated category. The function
$$
d(\sigma_1,\sigma_2)\coloneqq\sup_{E\in D\backslash\{0\}}\left\{\left|\log\frac{m_{\sigma_1}(E)}{m_{\sigma_2}(E)}\right|,\left|\phi^+_{\sigma_1}(E)-\phi^+_{\sigma_2}(E)\right|,\left|\phi^-_{\sigma_1}(E)-\phi^-_{\sigma_2}(E)\right|\right\}\in[0,+\infty]
$$
defines a generalized metric on $\Stab_\Lambda(\cD)$.
\end{defn}

For our purposes, it will be useful to distinguish between the contributions of the \emph{mass} component $m_\sigma$ and the \emph{phase} components $\phi_\sigma^\pm$ in the distance function:

\begin{notn}
\label{notn:distance-mass-phase}
We denote the \emph{mass-distance} and the \emph{phase-distance} as:
\begin{align*}
    d_\mass(\sigma_1,\sigma_2) & \coloneqq\sup_{E\in D\backslash\{0\}}\left\{\left|\log\frac{m_{\sigma_1}(E)}{m_{\sigma_2}(E)}\right|\right\}, \\
    d_\phase(\sigma_1,\sigma_2) & \coloneqq\sup_{E\in D\backslash\{0\}}\left\{\left|\phi^+_{\sigma_1}(E)-\phi^+_{\sigma_2}(E)\right|,\left|\phi^-_{\sigma_1}(E)-\phi^-_{\sigma_2}(E)\right|\right\}.
\end{align*}
Then $d(\sigma_1,\sigma_2)=\max\left\{d_\mass(\sigma_1,\sigma_2),d_\phase(\sigma_1,\sigma_2)\right\}$.
\end{notn}

This generalized metric induces a topology on the space of stability conditions. The main theorem of \cite{BriStab} asserts that $\Stab_\Lambda(\cD)$ can be endowed with the structure of a complex manifold.

\begin{thm}[\cite{BriStab}*{Theorem~7.1}]
The map
$$
\Stab_\Lambda(\cD)\rightarrow\Hom(\Lambda,\bC); \quad \sigma=(Z,P)\mapsto Z
$$
is a local homeomorphism, where $\Hom(\Lambda,\bC)$ is equipped with the linear topology. Therefore, $\Stab_\Lambda(\cD)$ is naturally a complex manifold of dimension $\rank(\Lambda)$.
\end{thm}

The space $\Stab_\Lambda(\cD)$ carries natural group actions by the group of autoequivalences and the universal cover $\widetilde{\GL^+(2,\bR)}$.
In what follows, we will focus exclusively on the action by the subgroup $\mathbb{C} \subseteq \widetilde{\GL^+(2,\mathbb{R})}$.
For $(x,y)\in\bR^2\cong\bC$, the action is given by:
$$
\sigma=(Z,P)\mapsto\sigma.(x,y)\coloneqq\left(e^{x+i\pi y}\cdot Z, P'\right), \quad \text{ where } P'(\phi)=P(\phi-y).
$$
Note that the $x$-component scales the masses and the $y$-component shifts the phases.
One can easily see that the $\bC$-actions are isometries with respect to the metric $d$.

\begin{defn}
\label{defn:bridgeland-dist-Cquotient}
The quotient metric on $\Stab(\cD)/\bC$ is defined by:
$$
\bar{d}(\bar{\sigma}_1,\bar{\sigma}_2)=\inf_{(x,y)\in\bC}d\left(\sigma_1,\sigma_2.(x,y)\right).
$$
Since every $\bC$-orbit is closed \cite{Woolf}*{Section~2}, this quotient metric is well-defined. The mass and phase distances are defined similarly as in Notation~\ref{notn:distance-mass-phase}:
\begin{align*}
    \bar{d}_\mass(\bar{\sigma}_1,\bar{\sigma}_2) & \coloneqq
    \inf_{x\in\bR}\left\{
    \sup_{E\in D\backslash\{0\}}\left\{\left|\log\frac{m_{\sigma_1}(E)}{m_{\sigma_2}(E)}-x\right|\right\}\right\}, \\
    \bar{d}_\phase(\bar{\sigma}_1,\bar{\sigma}_2) & \coloneqq
    \inf_{y\in\bR}\left\{
    \sup_{E\in D\backslash\{0\}}\left\{\left|\phi^+_{\sigma_1}(E)-\phi^+_{\sigma_2}(E)-y\right|,\left|\phi^-_{\sigma_1}(E)-\phi^-_{\sigma_2}(E)-y\right|\right\}\right\},
\end{align*}
with the relation $\bar{d}(\sigma_1,\sigma_2)=\max\left\{\bar{d}_\mass(\sigma_1,\sigma_2),\bar{d}_\phase(\sigma_1,\sigma_2)\right\}$.
\end{defn}

\begin{rmk}
For $\cD=\Db\Coh(X)$, the bounded derived category of coherent sheaves on $X$, the standard choice for $\Lambda$ is the numerical Grothendieck group $N(\Db\Coh(X))$. When $X$ is a curve, there is an isomorphism $N(\Db\Coh(X))=\bZ\oplus\bZ$, where the quotient map $K_0(\Db\Coh(X))\rightarrow N(\Db\Coh(X))$ sends each class to its rank and degree.
\end{rmk}

\subsection{Metric geometry preliminaries.}
\label{subsec:metric-geometry}
We recall the definitions of length spaces and geodesic spaces in this subsection.

\begin{defn}
Let $(X,d)$ be a metric space.
    \begin{itemize}
        \item The \emph{length} $L(\gamma)\in[0,\infty]$ of a curve $\gamma\colon[0,1]\rightarrow X$ is defined to be
        $$
        L(\gamma)\coloneqq\sup\left(\sum_{i=1}^k d\left(\gamma(t_{i-1}),\gamma(t_i)\right)\right)
        $$
        where the supremum is taken over all sequences $0=t_0\leq t_1\leq\cdots\leq t_k=1$.
        \item A metric space $(X,d)$ is called a \emph{length space} if for any $x,y\in X$,
        $$
        d(x,y)=\inf L(\gamma),
        $$
        where the infimum is taken over all curves $\gamma\colon[0,1]\rightarrow X$ with $\gamma(0)=x$ and $\gamma(1)=y$.
    \end{itemize}
\end{defn}

\begin{defn}
Let $(X,d)$ be a metric space.
\begin{itemize}
    \item A curve $\gamma\colon[0,1]\rightarrow X$ is called a \emph{geodesic} if $d(\gamma(t_1),\gamma(t_2))=|t_1-t_2|\cdot d(\gamma(0),\gamma(1))$ for all $t_1,t_2\in[0,1]$.
    \item The space  $(X,d)$ is a:
        \begin{itemize}
            \item \emph{geodesic space} if any two points are connected by at least one geodesic;
            \item \emph{uniquely geodesic space} if any two points are connected by exactly one geodesic.
        \end{itemize}
\end{itemize}
\end{defn}

\noindent Note that every geodesic space is a length space, since geodesics realize minimal lengths.

\begin{lemma}
\label{lemma:geodesiccriterion}
Let $(X,d)$ be a metric space and let $\gamma\colon[0,1]\rightarrow X$ be a curve without self-intersection. If for every $0\leq t_1<t_2<t_3\leq1$, the equality
$$
d(\gamma(t_1),\gamma(t_3))=d(\gamma(t_1),\gamma(t_2))+d(\gamma(t_2),\gamma(t_3))
$$
holds, then $\gamma$ admits a geodesic reparametrization.
\end{lemma}

\begin{proof}
Define the normalized length function
$$
F(t)\coloneqq\frac{d(\gamma(0),\gamma(t))}{d(\gamma(0),\gamma(1))},
$$
which is strictly increasing and continuous by the distance-additivity condition, with $F(0) = 0$ and $F(1) = 1$.
Let $G\colon[0,1]\rightarrow[0,1]$ be its inverse, and set
$$
\tilde{\gamma}(s)\coloneqq\gamma(G(s)).
$$
Then, for any $0\leq s_1<s_2\leq1$,
\begin{align*}
    d(\tilde{\gamma}(s_1),\tilde{\gamma}(s_2)) & = d(\tilde{\gamma}(0),\tilde{\gamma}(s_2)) - d(\tilde{\gamma}(0),\tilde{\gamma}(s_1)) \\
    & = d(\gamma(0),\gamma(G(s_2)))-d(\gamma(0),\gamma(G(s_1))) \\
    & = (F(G(s_2))-F(G(s_1)))\cdot d(\gamma(0),\gamma(1)) \\
    & = (s_2-s_1) \cdot d(\gamma(0),\gamma(1)).
\end{align*}
Thus, $\tilde{\gamma}$ is a geodesic.
\end{proof}

\subsection{The modified metrics on the upper half-plane.}
\label{subsec:metric-on-H}
Let $\mathbb{H} = \{(x,y) \in \mathbb{R}^2 \mid y > 0\}$ denote the upper half-plane, equipped with the standard Poincar\'e hyperbolic metric $d_\hyp$. 
The geodesics are either semi-circles centered on the real axis or vertical rays orthogonal to the real axis. 
The hyperbolic distance between two points $\tau_1, \tau_2 \in \mathbb{H}$ can be calculated as follows:
Let $t_1, t_2 \in \partial\overline{\mathbb{H}}=\bR\cup\{\infty\}$ be the endpoints of the unique geodesic through $\tau_1$ and $\tau_2$, ordered so that $t_i$ is closer to $\tau_i$. Then
$$
d_\hyp(\tau_1,\tau_2)=\log\frac{|\tau_1-t_2||\tau_2-t_1|}{|\tau_2-t_2||\tau_1-t_1|},
$$
where $|\cdot|$ denotes the standard Euclidean norm.

\begin{center}
\begin{tikzpicture}[thick]
\begin{scope}
    \clip (-2.5,0) rectangle (3,3);
    \draw (0,0) circle(2.5);
    \draw (-2.5,0) -- (2.5,0);
\end{scope}
    \draw
      (-4,0) edge[-latex] node[at end,right]{} (4,0)
      (-1.5,2)  node[scale=3](i){.} node[left]{$\tau_1$}
      (2,1.5)  node[scale=3](i){.} node[right]{$\tau_2$}
      (-2.5,0)  node[scale=3](i){.} node[below]{$t_1$}
      (2.5,0)  node[scale=3](i){.} node[below]{$t_2$}
    ;
\end{tikzpicture}
\end{center}

\begin{lemma}
The hyperbolic distance admits the following equivalent expressions:
$$
d_\hyp(\tau_1,\tau_2)=\sup_{s_1,s_2\in\bR}\left(\log\frac{|\tau_1-s_2||\tau_2-s_1|}{|\tau_2-s_2||\tau_1-s_1|}\right)=\sup_{s_1,s_2\in\bQ}\left(\log\frac{|\tau_1-s_2||\tau_2-s_1|}{|\tau_2-s_2||\tau_1-s_1|}\right).
$$
\end{lemma}

\begin{proof}
It suffices to show that the supremum is achieved when $s_1=t_1$ and $s_2=t_2$.
Let $\rho\in\PSL(2,\bR)$ be a hyperbolic isometry mapping $\tau_1,\tau_2$ to the imaginary axis with $\mathrm{Im}(\rho(\tau_1)) > \mathrm{Im}(\rho(\tau_2))$.
Since cross-ratios are preserved, we have
$$
\sup_{s_1,s_2\in\bR}\left(\log\frac{|\tau_1-s_2||\tau_2-s_1|}{|\tau_2-s_2||\tau_1-s_1|}\right)=\sup_{s_1,s_2\in\bR}\left(\log\frac{|\rho(\tau_1)-\rho(s_2)||\rho(\tau_2)-\rho(s_1)|}{|\rho(\tau_2)-\rho(s_2)||\rho(\tau_1)-\rho(s_1)|}\right)
$$
Since $\rho$ preserves the boundary $\partial\overline{\bH}=\bR\cup\{\infty\}$, it is clear that the supremum occurs when $\rho(s_1)=0$ and $\rho(s_2)=\infty$, i.e.~when $\rho(\tau_1),\rho(\tau_2),\rho(s_1),\rho(s_2)$ are on the same geodesic. This happens if and only if $s_1=t_1$ and $s_2=t_2$.
\end{proof}

We now introduce a modified metric on $\mathbb{H}$ that will be used in later sections.

\begin{defn}
\label{defn:H-metric-Z}
The metric $d_\bZ$ on the upper half-plane $\bH$ is defined by:
$$
d_\bZ(\tau_1,\tau_2)\coloneqq\sup_{s_1,s_2\in\bZ}\left(\log\frac{|\tau_1-s_2||\tau_2-s_1|}{|\tau_2-s_2||\tau_1-s_1|}\right).
$$
Note that it naturally extends to a metric on $\bH\cup(\bR\backslash\bZ)$ via the same formula.
\end{defn}

Proposition~\ref{prop:metric-dZ-geodesic-etc} will establish that $(\bH\cup(\bR\backslash\bZ),d_\bZ)$ is geodesic yet not uniquely geodesic.

\section{A stability space that is not a length space}
This section establishes that the stability space $(\Stab(\Db(\mathbb{P}^1)), d)$ with its canonical metric fails to be a length space: Section~\ref{subsec:3-1-stability-P1} reviews stability conditions on $\mathbb{P}^1$; Sections~\ref{subsec:3-2-dist-geom}--\ref{subsec:3-4-dist-alg} provide explicit distance formulas; Section~\ref{subsec:proof-of-thm} presents the proof of Theorem~\ref{thm:main}.

\subsection{Stability conditions on the projective line.}
\label{subsec:3-1-stability-P1}
The space of Bridgeland stability conditions on $\Db(\bP^1)$ has been completely characterized by Okada \cite{Okada}. In this subsection, we recall essential results from this work and introduce notation for parametrizing the stability space, which will be used throughout our analysis.

Stability conditions on $\Db(\bP^1)$ are of two distinct types:
\begin{itemize}
    \item \emph{Geometric stability conditions}, with respect to which all line bundles and skyscraper sheaves are stable.
    \item \emph{Algebraic stability conditions}, with respect to which only $\cO(k)$, $\cO(k+1)$, and their shifts are stable for some $k\in\bZ$.
\end{itemize}

\noindent\textbf{Geometric stability conditions.}
Denote by $\Stab^\Geo(\Db(\bP^1))$ the subset of geometric stability conditions in $\Stab(\Db(\bP^1))$. There is an isomorphism $\Stab^\Geo(\Db(\bP^1))\cong\bH\times\bC$, where the $\bC$-factor corresponds to the free $\mathbb{C}$-action on stability conditions.
A section of this action is determined by $\tau \in \mathbb{H}$ through:
\begin{itemize}
    \item The central charge is $Z(E)=-\deg(E)+\tau\cdot\rank(E)$.
    \item Skyscraper sheaves have phase $\phi(\cO_x)=1$.
    \item Line bundles have phases $0<\phi(\cO(n))=\frac{1}{\pi}\arg(\tau-n)<1$.
\end{itemize}

\begin{notn}
\label{notn:geometric}
The geometric stability conditions admit a parametrization by triples 
$$
(\tau, x, y) \in \mathbb{H} \times \mathbb{R}^2 \cong \mathbb{H} \times \mathbb{C}.
$$
For each such triple, let $\sigma_{\tau,x,y}$ denote the unique geometric stability condition satisfying:
\begin{itemize}
    \item The central charge is $Z_{\tau,x,y}(E)=e^{x+i\pi y}\left(-\deg(E)+\tau\cdot\rank(E)\right)$.
    \item Skyscraper sheaves have phase $\phi_{\tau,x,y}(\cO_x)=y+1$.
    \item Line bundles have phases $y<\phi_{\tau,x,y}(\cO(n))=y+\frac{1}{\pi}\arg(\tau-n)<y+1$.
\end{itemize}
Note that for each $\tau \in \mathbb{H}$, the stability conditions $\{\sigma_{\tau,x,y}:(x,y)\in\bR^2\}$ form a single $\mathbb{C}$-orbit.
We denote by $\bar{\sigma}_\tau$ the corresponding element in $\Stab^\Geo(\Db(\bP^1))/\bC\cong\bH$.
\end{notn}

\noindent\textbf{The subsets $X_k$.}
Beyond geometric stability conditions, define for each $k\in\bZ$ the subset
$$
X_k\coloneqq\left\{\sigma\in\Stab(\Db(\bP^1)):\cO(k)\text{ and }\cO(k+1)\text{ are }\sigma\text{-stable}\right\}\subseteq\Stab(\Db(\bP^1)).
$$
Each $X_k$ contains all geometric stability conditions. Moreover,  the algebraic stability conditions in distinct $X_k$ are disjoint. This yields the decomposition:
$$
\Stab(\Db(\bP^1))=\Stab^\Geo(\Db(\bP^1))\coprod\left(\coprod_{k\in\bZ}\left(X_k\backslash\Stab^\Geo(\Db(\bP^1))\right)\right).
$$

For each $k\in\bZ$, the set $X_k$ is isomorphic to $\bH\times\bC$, where the $\bC$-factor corresponds to the free $\mathbb{C}$-action on stability conditions. A section of this action is determined by $(\alpha,\beta)\in\bR\times\bR_{>0}\cong\bH$ through the following conditions:
\begin{itemize}
    \item Both $\cO(k)$ and $\cO(k+1)$ are stable.
    \item $Z(\cO(k))=1$ and $\phi(\cO(k))=0$.
    \item $Z(\cO(k+1))=e^{\alpha+i\pi\beta}$ and $\phi(\cO(k+1))=\beta>0$.
\end{itemize}

\begin{notn}
\label{notn:algebraic}
For each $k\in\bZ$, the stability conditions in $X_k$ admit a parametrization by quadruples
$$
(\alpha,\beta,x,y)\in\bR\times\bR_{>0}\times\bR^2\cong\bH\times\bC.
$$
For each such quadruple, let $\sigma_{\alpha,\beta,x,y}$ denote the unique stability condition satisfying:
\begin{itemize}
    \item Both $\cO(k)$ and $\cO(k+1)$ are $\sigma_{\alpha,\beta,x,y}$-stable.
    \item $Z_{\alpha,\beta,x,y}(\cO(k))=e^{x+i\pi y}$ and $\phi_{\alpha,\beta,x,y}(\cO(k))=y$.
    \item $Z_{\alpha,\beta,x,y}(\cO(k+1))=e^{x+\alpha+i\pi(y+\beta)}$ and $\phi_{\alpha,\beta,x,y}(\cO(k+1))=y+\beta$.
\end{itemize}
Note that for each $(\alpha,\beta)\in\bH$, the stability conditions $\{\sigma_{\alpha,\beta,x,y}:(x,y)\in\bR^2\}$ form a single $\mathbb{C}$-orbit.
We denote by $\bar{\sigma}_{\alpha,\beta}$ the corresponding element in $X_k/\bC\cong\bH$.
\end{notn}

Under this parametrization of $X_k$, the geometric stability conditions correspond precisely to the parameter subset
$$
\left\{(\alpha,\beta,x,y)\in\bR^2\times\bR^2:0<\beta<1\right\}.
$$
The coordinate transformation between this parametrization and the geometric parametrization in Notation~\ref{notn:geometric} is given by the following lemma, obtained directly by matching the central charges and phases of $\mathcal{O}(k)$ and $\mathcal{O}(k+1)$ under both descriptions.

\begin{lemma}
\label{lemma:change-of-var}
Let $\sigma$ be a geometric stability condition with $X_k$-coordinates $(\alpha, \beta, x_k, y_k) \in \mathbb{R}^2 \times \mathbb{R}^2$ under Notation~\ref{notn:algebraic}, where $0 < \beta < 1$. Then its geometric coordinates $(\tau, x, y)$ from Notation~\ref{notn:geometric} are determined by:
\begin{align*}
    \tau & = k+\frac{1}{1-e^{\alpha+i\pi\beta}}, \\
    x & = x_k + \log\left|1-e^{\alpha+i\pi\beta}\right|, \\
    y & = y_k + \frac{1}{\pi}\arg\left(1-e^{\alpha+i\pi\beta}\right).
\end{align*}
Here, the complex number $1-e^{\alpha+i\pi\beta}$ lies in the lower half-plane, with $\arg\left(1-e^{\alpha+i\pi\beta}\right)\in(-\pi,0)$.
\end{lemma}

\begin{rmk}
\label{rmk:boundary-parametrization}
Let $\overline{\Stab^\Geo(\Db(\bP^1))}\subseteq\Stab(\Db(\bP^1))$ 
denote the closure of the geometric stability conditions in the full stability space. 
The boundary
$$
\partial\overline{\Stab^\Geo(\Db(\bP^1))}=\overline{\Stab^\Geo(\Db(\bP^1))}\backslash\Stab^\Geo(\Db(\bP^1))
$$
consists of $\bZ$-many connected components, with each component contained in some $X_k$.

Using Notation~\ref{notn:geometric}, the boundary points admit the parametrization
$$
(\tau,x,y)\in(\bR\backslash\bZ)\times\bR^2=\coprod_{k\in\bZ}(k,k+1)\times\bR^2.
$$
For $\tau\in(k,k+1)$, the corresponding stability condition is uniquely determined by:
\begin{itemize}
    \item The central charge is $Z_{\tau,x,y}(E)=e^{x+i\pi y}\left(-\deg(E)+\tau\cdot\rank(E)\right)$.
    \item Skyscraper sheaves have phase $\phi_{\tau,x,y}(\cO_x)=y+1$.
    \item Line bundles have phases $\phi_{\tau,x,y}(\cO(n))=y+1$ if $n\geq k+1$, and $\phi_{\tau,x,y}(\cO(n))=y$ if $n\leq k$.
\end{itemize}
Such a stability condition lies in $X_k$, with $\beta=1$ in Notation~\ref{notn:algebraic}. The components of $(\tau,x,y)\in(k,k+1)\times\bR^2$ and $(\alpha,1,x_k,y_k)\in X_k$ are related as in Lemma~\ref{lemma:change-of-var}:
\begin{align*}
    \tau & = k+\frac{1}{1+e^{\alpha}}, \\
    x & = x_k + \log\left(1+e^{\alpha}\right), \\
    y & = y_k.
\end{align*}
\end{rmk}

\begin{defn}
\label{defn:projection-to-boundary}
Define the projection map
$$
p\colon\Stab(\Db(\bP^1))\rightarrow\overline{\Stab^\Geo(\Db(\bP^1))}
$$
as follows:
\begin{itemize}
    \item For $\sigma\in\Stab^\Geo(\Db(\bP^1))$, let $p(\sigma)\coloneqq\sigma$.
    \item For $\sigma_{\alpha,\beta,x,y}\in X_k\backslash\Stab^\Geo(\Db(\bP^1))$, let $p(\sigma_{\alpha,\beta,x,y})\coloneqq\sigma_{\alpha,1,x,y}$.
\end{itemize}
Similarly, define
$$
\bar{p}\colon\Stab(\Db(\bP^1))/\bC\rightarrow\overline{\Stab^\Geo(\Db(\bP^1))}/\bC
$$
by:
\begin{itemize}
    \item For $\bar\sigma\in\Stab^\Geo(\Db(\bP^1))/\bC$, let $\bar{p}(\bar\sigma)\coloneqq\bar\sigma$,
    \item For $\bar\sigma_{\alpha,\beta}\in (X_k/\bC)\backslash(\Stab^\Geo(\Db(\bP^1))/\bC)$, let $\bar{p}(\bar\sigma_{\alpha,\beta})\coloneqq\bar\sigma_{\alpha,1}$.
\end{itemize}
These maps collapse algebraic stability conditions to the boundary of the geometric stability space while leaving geometric stability conditions unchanged.
\end{defn}

\subsection{Distance between geometric stability conditions.}
\label{subsec:3-2-dist-geom}
In this subsection, we derive explicit formulas for the distances between geometric stability conditions.

\begin{lemma}
The metrics on $\Stab(\Db(\bP^1))$ and $\Stab(\Db(\bP^1))/\bC$ (cf.~Definitions~\ref{defn-bridgeland-dist} and \ref{defn:bridgeland-dist-Cquotient}) can be computed by taking the supremum over only line bundles and skyscraper sheaves on $\bP^1$.
\end{lemma}

\begin{proof}
The claim follows from the fact that every coherent sheaf on $\mathbb{P}^1$ decomposes into a finite direct sum of line bundles and skyscraper sheaves.
\end{proof}

The distance between geometric stability conditions is given by the following lemma.

\begin{lemma}
\label{lemma:geometric-distance}
Let $\sigma_1=\sigma_{\tau_1,x_1,y_1}$ and $\sigma_2=\sigma_{\tau_1,x_1,y_1}$ be two geometric stability conditions (see Notation~\ref{notn:geometric}). Their distance is given by
$$
d(\sigma_1,\sigma_2)=\max\left\{d_\mass(\sigma_1,\sigma_2),d_\phase(\sigma_1,\sigma_2)\right\},
$$
where
\begin{align}
    d_\mass(\sigma_1,\sigma_2) & =\max\left\{\sup_{n\in\bZ}\left\{\left|\log\frac{|\tau_1-n|}{|\tau_2-n|}+(x_1-x_2)\right|\right\},|x_1-x_2|\right\}, \label{eqn:mass-distance-geom}\tag{3.a}\\
    d_\phase(\sigma_1,\sigma_2) & =\max\left\{\sup_{n\in\bZ}\left\{\left|\frac{1}{\pi}
\arg\left(\frac{\tau_1-n}{\tau_2-n}\right)+(y_1-y_2)\right|\right\},|y_1-y_2|\right\}. \nonumber
\end{align}
Here, $\arg\left(\frac{\tau_1-n}{\tau_2-n}\right)$ takes values in $(-\pi,\pi)$.
\end{lemma}

\begin{proof}
The claim follows from the fact that all line bundles and skyscraper sheaves are stable with respect to geometric stability conditions. 
\end{proof}

\begin{rmk}
\label{rmk:mass-distance-extend-closure}
The mass-distance formula (\ref{eqn:mass-distance-geom}) for $d_\mass$ extends naturally to the closure $\overline{\Stab^\Geo(\Db(\bP^1))}\cong(\bH\cup(\bR\backslash\bZ))\times\bR^2$, 
since line bundles and skyscraper sheaves remain semistable for stability conditions on the boundary of $\Stab^\Geo(\Db(\bP^1))$.
In contrast, the formula for the phase-distance $d_\phase$ does not admit a direct extension to the closure due to the multi-valuedness of $\arg$.
\end{rmk}

The following lemma will be useful for computing the quotient metric on $\Stab(\Db(\bP^1))/\bC$.

\begin{lemma}
\label{lemma:supinf}
Let $\{A_\alpha\}$ be a set of real numbers such that
$$
\sup_\alpha A_\alpha\geq0\geq\inf_\alpha A_\alpha.
$$
Then
$$
\inf_{\lambda\in\bR}\left\{\max\left\{\sup_\alpha\left\{|A_\alpha+\lambda|\right\},|\lambda|\right\}\right\}=\frac{1}{2}\left(\sup_\alpha A_\alpha-\inf_\alpha A_\alpha\right).
$$
\end{lemma}

\begin{proof}
It is clear that 
$$
\max\left\{\sup_\alpha\left\{|A_\alpha+\lambda|\right\}\right\}\geq\frac{1}{2}\left(\sup_\alpha A_\alpha-\inf_\alpha A_\alpha\right) \quad \text{ for any } \lambda\in\bR.
$$
Therefore, it suffices to show that the infimum can be obtained by some $\lambda\in\bR$. Let
$$
\lambda=-\frac{1}{2}\left(\sup_\alpha A_\alpha+\inf_\alpha A_\alpha\right).
$$
Then
$$
|A_\alpha+\lambda|\leq\frac{1}{2}\left(\sup_\alpha A_\alpha-\inf_\alpha A_\alpha\right) \quad \text{ for any } \alpha,
$$
and 
$$
|\lambda|=\frac{1}{2}\left|\sup_\alpha A_\alpha+\inf_\alpha A_\alpha\right|\leq\frac{1}{2}\left(\sup_\alpha A_\alpha-\inf_\alpha A_\alpha\right)
$$
since $\sup_\alpha A_\alpha\geq0\geq\inf_\alpha A_\alpha$.
\end{proof}

The quotient metric on $\Stab^\Geo(\Db(\bP^1))/\bC$ admits the following explicit expression.

\begin{prop}
\label{prop:geometric-distance/C}
For any two geometric stability conditions modulo scaling $\bar{\sigma}_{\tau_1}, \bar{\sigma}_{\tau_2} \in \Stab^\Geo(\Db(\mathbb{P}^1))/\mathbb{C} \cong \mathbb{H}$, their distance in the quotient metric is given by:
$$
\bar{d}(\bar{\sigma}_{\tau_1},\bar{\sigma}_{\tau_2})=\max\left\{\frac{1}{2}d_\bZ(\tau_1,\tau_2),\frac{1}{2\pi}\left(\sup_{n\in\bZ}\left\{\arg\left(\frac{\tau_1-n}{\tau_2-n}\right)\right\}-\inf_{m\in\bZ}\left\{\arg\left(\frac{\tau_1-m}{\tau_2-m}\right)\right\}\right)\right\}.
$$
Here, $d_\bZ$ is the metric on $\bH$ from Definition~\ref{defn:H-metric-Z}, and the arguments $\arg$ take values in $(-\pi,\pi)$.
\end{prop}

\begin{proof}
It is not hard to see that
$$
\sup_{n\in\bZ}\left\{\log\frac{|\tau_1-n|}{|\tau_2-n|}\right\}\geq0\geq\inf_{m\in\bZ}\left\{\log\frac{|\tau_1-m|}{|\tau_2-m|}\right\}
$$
and
$$
\sup_{n\in\bZ}\left\{\log\frac{|\tau_1-n|}{|\tau_2-n|}\right\}-\inf_{m\in\bZ}\left\{\log\frac{|\tau_1-m|}{|\tau_2-m|}\right\}=\sup_{n,m\in\bZ}\left(\log\frac{|\tau_1-n||\tau_2-m|}{|\tau_2-n||\tau_1-m|}\right)=d_\bZ(\tau_1,\tau_2).
$$
Therefore, by Lemma~\ref{lemma:supinf}, the mass-distance is given by
\begin{equation}\label{eqn:mass-distance-geo-/C}\tag{3.b}
\bar{d}_\mass(\bar{\sigma}_{\tau_1},\bar{\sigma}_{\tau_2})=\frac{1}{2}d_\bZ(\tau_1,\tau_2).
\end{equation}
The phase-distance is given by
$$
\bar{d}_\phase(\bar{\sigma}_1,\bar{\sigma}_2)=\frac{1}{2\pi}\left(\sup_{n\in\bZ}\left\{\arg\left(\frac{\tau_1-n}{\tau_2-n}\right)\right\}-\inf_{m\in\bZ}\left\{\arg\left(\frac{\tau_1-m}{\tau_2-m}\right)\right\}\right).
$$
This follows again from Lemma~\ref{lemma:supinf}, and the simple observation that
$$
\sup_{n\in\bZ}\left\{\arg\left(\frac{\tau_1-n}{\tau_2-n}\right)\right\}\geq0\geq\inf_{m\in\bZ}\left\{\arg\left(\frac{\tau_1-m}{\tau_2-m}\right)\right\}.
$$
\end{proof}

\begin{rmk}
Note that the phase-distance satisfies the uniform bound
$$
\bar{d}_\phase(\bar{\sigma}_1,\bar{\sigma}_2)=\frac{1}{2\pi}\left(\sup_{n\in\bZ}\left\{\arg\left(\frac{\tau_1-n}{\tau_2-n}\right)\right\}-\inf_{m\in\bZ}\left\{\arg\left(\frac{\tau_1-m}{\tau_2-m}\right)\right\}\right)\leq\frac{1}{2\pi}\left(\pi-(-\pi)\right)=1.
$$
Therefore, one has
$$
\bar{d}(\bar{\sigma}_{\tau_1},\bar{\sigma}_{\tau_2})=\frac{1}{2}d_\bZ(\tau_1,\tau_2) \quad \text{ whenever } d_\bZ(\tau_1,\tau_2)\geq2.
$$
\end{rmk}

\begin{rmk}
For a smooth projective curve $X$ of genus $g(X)\geq1$, the space of stability conditions $\Stab(\Db(X))\cong\bH\times\bC$ consists of only geometric stability conditions \cite{Macri}*{Theorem~2.7}.
In this case, the quotient metric $\bar{d}$ on $\Stab(\Db(X))/\bC\cong\bH$ is precisely half the hyperbolic metric \cite{Woolf}*{Proposition~4.1}. 
The main reason we obtain a different metric in the case of $\bP^1$ is that the phases of semistable objects are dense in $\bR$ for $g(X)\geq1$, while for $\bP^1$, this density fails.
\end{rmk}

\subsection{Distance between geometric and algebraic stability conditions.}
\label{subsec:geometry-algebraic-distance}
To compute the distance between a geometric stability condition and an algebraic stability condition, we first recall the following useful fact.

\begin{lemma}[\cite{Okada}*{Lemma~3.1}]
In $\Db(\bP^1)$, there are exact triangles:
\begin{center}
\begin{tabular}{ r c l l }
 $\cO(k+1)^{\oplus(n-k)}\rightarrow$ & $\cO(n)$ & $\rightarrow\cO(k)^{\oplus(n-k-1)}[1]\xrightarrow{+1}$ & if $n> k+1$  \\ 
 $\cO(k+1)^{\oplus(k-n)}[-1]\rightarrow$ & $\cO(n)$ & $\rightarrow\cO(k)^{\oplus(k-n+1)}\xrightarrow{+1}$ & if $n<k$ \\  
 $\cO(k+1)\rightarrow$ & $\cO_x$ & $\rightarrow\cO(k)[1]\xrightarrow{+1}$
\end{tabular}
\end{center}
Moreover, these exact triangles give the Harder--Narasimhan filtrations of line bundles $\cO(n)$ (for $n\neq k,k+1$) and skyscraper sheaves $\cO_x$ with respect to algebraic stability conditions in $X_k\backslash\Stab^\Geo(\Db(\bP^1))$.
\end{lemma}

For algebraic stability conditions in $X_k\backslash\Stab^\Geo(\Db(\bP^1))\cong\{(\alpha,\beta,x,y):\beta\geq1\}$ (see Notation~\ref{notn:algebraic}), the masses and phases of line bundles and skyscraper sheaves are summarized in Table~\ref{table:algebraic}.

\begin{table}[ht]
\centering
\begin{tabular}{ c|c|c|c } 
(for $\beta\geq1$) & $m_{\alpha,\beta,x,y}$ & $\phi^+_{\alpha,\beta,x,y}$ & $\phi^-_{\alpha,\beta,x,y}$ \\
\hline\hline
$\cO(n)$ ($n<k$) & $((k-n)e^{\alpha}+(k-n+1))e^x$ & $\beta+y-1$ & $y$ \\
\hline
$\cO(k)$ & $e^x$ & $y$ & $y$ \\
\hline
$\cO(k+1)$ & $e^{\alpha+x}$  & $\beta+y$ & $\beta+y$ \\
\hline
$\cO(n)$ ($n>k+1$) & $((n-k)e^{\alpha}+(n-k-1))e^x$ & $\beta+y$ & $y+1$ \\
\hline
$\cO_x$ & $(e^{\alpha}+1)e^x$ & $\beta+y$ & $y+1$
\end{tabular}
\caption{Masses and phases of objects with respect to an algebraic stability condition.}
\label{table:algebraic}
\end{table}

\begin{prop}
\label{prop:mass-projection}
Let $\sigma_1$ and $\sigma_2$ be two stability conditions on $\Db(\bP^1)$. Then
$$
d_\mass(\sigma_1,\sigma_2)=d_\mass(p(\sigma_1),p(\sigma_2)) \quad \text{ and } \quad \bar{d}_\mass(\bar{\sigma}_1,\bar{\sigma}_2)=\bar{d}_\mass(\bar{p}(\bar{\sigma}_1),\bar{p}(\bar{\sigma}_2)).
$$
Here, $p$ and $\bar{p}$ are the projections contracting algebraic stability conditions to the boundary of the geometric stability space (see Definition~\ref{defn:projection-to-boundary}).

Then, the mass-distances can be computed using formulae~(\ref{eqn:mass-distance-geom}) and (\ref{eqn:mass-distance-geo-/C}), as these extend to the closure of the geometric stability conditions (see Remark~\ref{rmk:mass-distance-extend-closure}).
\end{prop}

\begin{proof}
The result follows from the observation that the masses $m_{\alpha,\beta,x,y}$ of line bundles and skyscraper sheaves are independent of $\beta$ for algebraic stability conditions, i.e.~when $\beta\geq1$.
\end{proof}

Let $\sigma_1$ be a geometric stability condition and $\sigma_2$ an algebraic stability condition. There exists a unique $k\in\bZ$ such that both $\sigma_1$ and $\sigma_2$ lie in $X_k$. 
We now compute the distance between them using their coordinates from Notation~\ref{notn:algebraic}:
$$
(\alpha_1,\beta_1,x_1,y_1),(\alpha_2,\beta_2,x_2,y_2)\in\bH\times\bR^2.
$$
The mass-distance can be computed directly via Proposition~\ref{prop:mass-projection}. For the phase-distance, we additionally provide the masses and phases of line bundles and skyscraper sheaves with respect to geometric stability conditions (i.e.~for $0<\beta<1$) in the following table.

\begin{table}[ht]
\centering
\begin{tabular}{ c|c|c } 
(for $0<\beta<1$)& $m_{\alpha,\beta,x,y}$ & $\phi^\pm_{\alpha,\beta,x,y}$  \\
\hline\hline
$\cO(n)$ ($n<k$) & $|(k-n)e^{\alpha+i\pi\beta}-(k-n+1)|e^x$ & $\beta+y-1<\phi(\cO(n))<y$ \\
\hline
$\cO(k)$ & $e^x$ & $y$\\
\hline
$\cO(k+1)$ & $e^{\alpha+x}$ & $\beta+y$\\
\hline
$\cO(n)$ ($n>k+1$) & $|(n-k)e^{\alpha+i\pi\beta}-(n-k-1)|e^x$ & $\beta+y<\phi(\cO(n))<y+1$ \\
\hline
$\cO_x$ & $|e^{\alpha+i\pi\beta}-1|e^x$ & $\beta+y<\phi(\cO_x)<y+1$
\end{tabular}
\caption{Masses and phases of objects with respect to a geometric stability condition.}
\label{table:geometric}
\end{table}

\begin{lemma}
\label{lemma:geometric-algebraic-distance}
Let $\sigma_{\alpha_1,\beta_1,x_1,y_1}$ be a geometric stability condition and $\sigma_{\alpha_2,\beta_2,x_2,y_2}$ an algebraic stability condition in $X_k$ (therefore, $0<\beta_1<1\leq\beta_2$). Then
\begin{align*}
d_\mass\left(\sigma_{\alpha_1,\beta_1,x_1,y_1},\sigma_{\alpha_2,\beta_2,x_2,y_2}\right) & = d_\mass\left(\sigma_{\tilde{\tau}_1,\tilde{x}_1,\tilde{y}_1},\sigma_{\tilde{\tau}_2,\tilde{x}_2,\tilde{y}_2}
\right),\\
d_\phase\left(\sigma_{\alpha_1,\beta_1,x_1,y_1},\sigma_{\alpha_2,\beta_2,x_2,y_2}\right) & =\max\left\{|(\beta_1-\beta_2)+(y_1-y_2)|,|y_1-y_2|\right\},
\end{align*}
where
\begin{center}
\begin{tabular}{ ll } 
$\tilde{\tau}_1=k+\frac{1}{1-e^{\alpha_1+i\pi\beta_1}}$, & $\tilde{\tau}_2=k+\frac{1}{1+e^{\alpha_2}}$, \\
$\tilde{x}_1=x_1 + \log\left|1-e^{\alpha_1+i\pi\beta_1}\right|$, & $\tilde{x}_2=x_2 + \log\left(1+e^{\alpha_2}\right)$, \\
$\tilde{y}_1=y_1 + \frac{1}{\pi}\arg\left(1-e^{\alpha_1+i\pi\beta_1}\right)$, & $\tilde{y}_2=y_2$
\end{tabular}
\end{center}
are coordinates in $(\tau,x,y)\in(\bH\cup(\bR\backslash\bZ))\times\bR^2\cong\overline{\Stab^\Geo(\Db(\bP^1))}$.
\end{lemma}

\begin{proof}
The mass-distance equality follows from Proposition~\ref{prop:mass-projection}, combined with the coordinate transformation between the $(\alpha,\beta,x,y)$ and $(\tau,x,y)$ systems given in Lemma~\ref{lemma:change-of-var} (along with Remark~\ref{rmk:boundary-parametrization} for the boundary case). 
For the phase-distance, we compute it explicitly by evaluating the phase differences using Tables~\ref{table:algebraic} and \ref{table:geometric}.
\end{proof}

\begin{prop}
\label{prop:geometric-algebraic-distance/C}
Let $\bar{\sigma}_{\alpha_1,\beta_1}$ and $\bar{\sigma}_{\alpha_2,\beta_2}$ be two elements of $X_k/\bC\cong\bR\times\bR_{>0}$.
Suppose $0<\beta_1<1\leq\beta_2$, i.e.~$\bar{\sigma}_{\alpha_1,\beta_1}$ is geometric and $\bar{\sigma}_{\alpha_2,\beta_2}$ is algebraic. Then their distance is given by
$$
\bar{d}(\bar{\sigma}_{\alpha_1,\beta_1},\bar{\sigma}_{\alpha_2,\beta_2})=\frac{1}{2}\max\left\{d_\bZ\left(k+\frac{1}{1-e^{\alpha_1+i\pi\beta_1}},k+\frac{1}{1+e^{\alpha_2}}\right),\beta_2-\beta_1\right\}.
$$
\end{prop}

\begin{proof}
The mass-distance is again determined by Proposition~\ref{prop:mass-projection}. For the phase-distance, we combine the pre-quotient computation from Lemma~\ref{lemma:geometric-algebraic-distance} with Lemma~\ref{lemma:supinf}.
\end{proof}

\subsection{Distance between two algebraic stability conditions.}
\label{subsec:3-4-dist-alg}
We compute the distance between algebraic stability conditions in this subsection. Although not required for the proof of our main theorem, these results are included for completeness.

\begin{prop}
\label{prop:alg-stab-dist}
Let $\sigma_1=\sigma_{\alpha_1,\beta_1,x_1,y_1}\in X_k$ and $\sigma_2=\sigma_{\alpha_2,\beta_2,x_2,y_2}\in X_\ell$ be two algebraic stability conditions.
\begin{enumerate}[label=(\roman*)]
    \item\label{item:alg-stab-dist-same-chamber}
    Suppose $k=\ell$, i.e.~$\sigma_1$ and $\sigma_2$ are in the same algebraic chamber. Then
\begin{align*}
d_\mass\left(\sigma_{\alpha_1,\beta_1,x_1,y_1},\sigma_{\alpha_2,\beta_2,x_2,y_2}\right) & =\max\left\{|(\alpha_1-\alpha_2)+(x_1-x_2)|,|x_1-x_2|\right\} \\
d_\phase\left(\sigma_{\alpha_1,\beta_1,x_1,y_1},\sigma_{\alpha_2,\beta_2,x_2,y_2}\right) & =\max\left\{|(\beta_1-\beta_2)+(y_1-y_2)|,|y_1-y_2|\right\} \\
\bar{d}_\mass(\bar{\sigma}_{\alpha_1,\beta_1},\bar{\sigma}_{\alpha_2,\beta_2}) & = \frac{1}{2}|\alpha_1-\alpha_2| \\
\bar{d}_\phase(\bar{\sigma}_{\alpha_1,\beta_1},\bar{\sigma}_{\alpha_2,\beta_2}) & = \frac{1}{2}|\beta_1-\beta_2|
\end{align*}
    \item Suppose $k\neq\ell$, say $k<\ell$ without loss of generality. Then
\begin{align*}
d_\mass\left(\sigma_{\alpha_1,\beta_1,x_1,y_1},\sigma_{\alpha_2,\beta_2,x_2,y_2}\right) & =\max\{
\left|\log\left((\ell-k)(1+e^{\alpha_1})-1\right)+(x_1-x_2)\right|, \\
& \quad\qquad\ \ \left|\log\left((\ell-k+1)(1+e^{\alpha_1})-1\right)-\alpha_2+(x_1-x_2)\right|,\\
& \quad\qquad\ \ \left|-\log\left((\ell-k)(1+e^{\alpha_2})+1\right)+(x_1-x_2)\right|,\\
& \quad\qquad\ \ \left|\alpha_1-\log\left((\ell-k-1)(1+e^{\alpha_2})+1\right)+(x_1-x_2)\right|\}\\
d_\phase\left(\sigma_{\alpha_1,\beta_1,x_1,y_1},\sigma_{\alpha_2,\beta_2,x_2,y_2}\right) & =\max\left\{|\beta_1+(y_1-y_2)|,|(1-\beta_2)+(y_1-y_2))|\right\} \\
\bar{d}_\mass(\bar{\sigma}_{\alpha_1,\beta_1},\bar{\sigma}_{\alpha_2,\beta_2}) & = \frac{1}{2}d_\bZ\left(k+\frac{1}{1+e^{\alpha_1}},\ell+\frac{1}{1+e^{\alpha_2}}\right)\\
\bar{d}_\phase(\bar{\sigma}_{\alpha_1,\beta_1},\bar{\sigma}_{\alpha_2,\beta_2}) & = \frac{1}{2}(\beta_1+\beta_2-1)
\end{align*}
\end{enumerate}
\end{prop}

\begin{proof}
The mass-distances are computed using Proposition~\ref{prop:mass-projection}, and the phase-distances follow from Table~\ref{table:algebraic}.
\end{proof}

Note that the case when $k=\ell$ was previously treated in \cite{Kikuta}*{Section~3.2}.

\subsection{Proof of Theorem~\ref{thm:main}.}
\label{subsec:proof-of-thm}
In this subsection, we prove Theorem~\ref{thm:main}, establishing that neither $(\Stab(\Db(\bP^1)),d)$ nor its quotient $(\Stab(\Db(\bP^1))/\bC,\bar{d})$ is a length space.

Before presenting the full proof, let us outline the key idea.
First, we choose a geometric stability condition $\sigma_1$ that is far away from the boundary $\partial\overline{\Stab^\Geo(\Db(\bP^1))}$, along with a boundary stability condition $\sigma_2\in\partial\overline{\Stab^\Geo(\Db(\bP^1))}$ such that:
\begin{itemize}
    \item $d(\sigma_1,\sigma_2)$ realizes the infimum distance from $\sigma_1$ to the boundary:
    $$
    d(\sigma_1,\sigma_2)=\inf\left\{d(\sigma_1,\tilde{\sigma}):\tilde{\sigma}\in\partial\overline{\Stab^\Geo(\Db(\bP^1))}\right\},
    $$
    \item the mass distance dominates the phase distance:  $d_\mass(\sigma_1,\sigma_2)\gg d_\phase(\sigma_1,\sigma_2)$.
\end{itemize}

Let $\sigma_2=\sigma_{\alpha,1,x,y}\in X_k$. 
We then consider a small deformation of it into the algebraic chamber, say $\sigma_3=\sigma_{\alpha,1+\epsilon,x,y}\in X_k$. Then:
\begin{itemize}
    \item $\sigma_3$ has a positive distance from the boundary:
        $$
        \inf\left\{d(\sigma_3,\tilde{\sigma}):\tilde{\sigma}\in\partial\overline{\Stab^\Geo(\Db(\bP^1))}\right\}\geq\frac{\epsilon}{2}>0,
        $$
    \item this deformation preserves the distance to $\sigma_1$:
    $d(\sigma_1,\sigma_2)=d(\sigma_1,\sigma_3)$: 
    the mass-distance remains identical by Proposition~\ref{prop:mass-projection}, and the phase-distance increases by at most $\epsilon$ (Lemma~\ref{lemma:geometric-algebraic-distance}), remaining dominated by the mass distance.
\end{itemize}

Any path $\gamma$ connecting $\sigma_1$ to $\sigma_3$ must cross the boundary $\partial\overline{\Stab^\Geo(\Db(\bP^1))}$. A simple length estimate gives
$$
L(\gamma)\geq d(\sigma_1,\sigma_3) + \frac{\epsilon}{2},
$$
which shows that the distance between $\sigma_1$ and $\sigma_3$ cannot be realized as the infimum of lengths of connecting paths.

\begin{proof}[Proof of Theorem~\ref{thm:main}]
Fix an explicit geometric stability condition
$$
\sigma_1=\sigma_{\left(\frac{1}{2}+10i,0,0\right)},
$$
where $\left(\frac{1}{2}+10i,0,0\right)\in\bH\times\bR^2$ are its coordinates following Notation~\ref{notn:geometric}. 
By Lemma~\ref{lemma:geometric-distance}, the infimum $\inf\left\{d(\sigma_1,\tilde{\sigma}):\tilde{\sigma}\in\partial\overline{\Stab^\Geo(\Db(\bP^1))}\right\}$ is achieved at
$$
\sigma_2=\sigma_{\left(\frac{1}{2},\frac{1}{2}\log\sqrt{401},0\right)}\in\partial\overline{\Stab^\Geo(\Db(\bP^1))}\cap X_0,
$$
where $\left(\frac{1}{2},\frac{1}{2}\log\sqrt{401},0\right)\in(0,1)\times\bR^2$ are its coordinates as defined in Notation~\ref{notn:geometric} and Remark~\ref{rmk:boundary-parametrization}.
A direct calculation yields:
\begin{align*}
    d(\sigma_1,\sigma_2) & = \inf\left\{d(\sigma_1,\tilde{\sigma}):\tilde{\sigma}\in\partial\overline{\Stab^\Geo(\Db(\bP^1))}\right\} = \frac{1}{2}\log\sqrt{401} \\
    d_\mass(\sigma_1,\sigma_2) & = \frac{1}{2}\log\sqrt{401} \\
    d_\phase(\sigma_1,\sigma_2) & < 1
\end{align*}

Next, we express $\sigma_2\in X_0$ using the algebraic coordinates from Notation~\ref{notn:algebraic}. By Remark~\ref{rmk:boundary-parametrization}, we have:
$$
\sigma_2=\sigma_{\left(0,1,\frac{1}{2}\log\sqrt{401}-\log2,0\right)}\in X_0.
$$
We then deform $\sigma_2$ slightly into the algebraic chamber by increasing its $\beta$-component by $0.1$, obtaining:
$$
\sigma_3=\sigma_{\left(0,1.1,\frac{1}{2}\log\sqrt{401}-\log2,0\right)}\in X_0.
$$
The distance between $\sigma_1$ and $\sigma_3$ satisfies:
\begin{itemize}
    \item \emph{Mass-distance}: By Proposition~\ref{prop:mass-projection},
    $d_\mass(\sigma_1,\sigma_3)=d_\mass(\sigma_1,\sigma_2)=\frac{1}{2}\log\sqrt{401}$.
    \item \emph{Phase-distance}:
    $d_\phase(\sigma_1,\sigma_3)\leq d_\phase(\sigma_1,\sigma_2)+d_\phase(\sigma_2,\sigma_3)<1.1<\frac{1}{2}\log\sqrt{401}$.
\end{itemize}
Therefore,
$$
d(\sigma_1,\sigma_3)=\max\left\{d_\mass(\sigma_1,\sigma_3),d_\phase(\sigma_1,\sigma_3)\right\}=d(\sigma_1,\sigma_2)=\frac{1}{2}\log\sqrt{401}.
$$
By Lemma~\ref{lemma:geometric-algebraic-distance}, $\sigma_3$ has a positive distance from the boundary:
\begin{align*}
    \inf\left\{d(\sigma_3,\tilde{\sigma}):\tilde{\sigma}\in\partial\overline{\Stab^\Geo(\Db(\bP^1))}\right\} & \geq
    \inf\left\{d_\phase(\sigma_3,\tilde{\sigma}):\tilde{\sigma}\in\partial\overline{\Stab^\Geo(\Db(\bP^1))}\right\} \\
    & \geq0.05.
\end{align*}

Since every path $\gamma$ connecting $\sigma_1$ to $\sigma_3$ must intersect the boundary $\partial\overline{\Stab^\Geo(\Db(\bP^1))}$ at some point $\sigma^*$, we obtain the following length estimate:
\begin{align*}
    L(\gamma) & \geq d(\sigma_1,\sigma^*) + d(\sigma^*, \sigma_3) \\
    & \geq d(\sigma_1,\sigma_2) + 0.05 = d(\sigma_1,\sigma_3) + 0.05.
\end{align*}
This proves that $\Stab(\Db(\bP^1))$ is not a length space, as the distance between $\sigma_1$ and $\sigma_3$ cannot be realized as the infimum of lengths of paths connecting them.

The proof that $\Stab(\Db(\bP^1))/\bC$ is not a length space follows analogously. Consider the geometric stability condition
$$
\bar{\sigma}_1=\bar{\sigma}_{\frac{1}{2}+10i}\in\Stab^\Geo(\Db(\bP^1))/\bC.
$$
The minimal distance to the boundary $\partial\overline{\Stab^\Geo(\Db(\bP^1))}/\bC$ is achieved by $\bar\sigma_2=\bar\sigma_{\frac{1}{2}}$, with the following explicit values:
\begin{align*}
    \bar{d}(\bar\sigma_1,\bar\sigma_2) & = \inf\left\{\bar{d}(\bar\sigma_1,\bar{\tilde{\sigma}}):\bar{\tilde{\sigma}}\in\partial\overline{\Stab^\Geo(\Db(\bP^1))}/\bC\right\} = \frac{1}{2}\log\sqrt{401} \\
    \bar{d}_\mass(\bar\sigma_1,\bar\sigma_2) & = \frac{1}{2}\log\sqrt{401} \\
    \bar{d}_\phase(\bar\sigma_1,\bar\sigma_2) & < 1
\end{align*}
Using the algebraic coordinates from Notation~\ref{notn:algebraic}, we have:
$$
\bar\sigma_2=\bar\sigma_{(0,1)}\in X_0.
$$
Consider a small deformation into the algebraic chamber:
$$
\bar\sigma_3=\bar\sigma_{(0,1.1)},
$$
obtained by incrementing the $\beta$-parameter.
The identical argument now shows that $\bar{d}(\bar{\sigma}_1, \bar{\sigma}_3)$ cannot be achieved as the infimum of path lengths between them. We conclude that $\Stab(\Db(\bP^1))/\bC$ fails to be a length space.
\end{proof}

\section{Two modified metrics}
In this section, we introduce and analyze two modified metrics on stability spaces, with particular emphasis on their properties for the case of projective line.

\subsection{Mass metric.}
\label{subsec:mass-metric}
Recall that the original distance function (see Notation~\ref{notn:distance-mass-phase}) is given by:
$$
d(\sigma_1,\sigma_2)=\max\left\{d_\mass(\sigma_1,\sigma_2),d_\phase(\sigma_1,\sigma_2)\right\}.
$$
The \emph{mass metric} arises as a limiting metric where the phase contribution goes to zero.
This type of degeneration appears also in \cite{FFHKL}*{Section~2.1} and \cite{BaeLee}*{Remark~10.13}.
The underlying space for the mass metric is defined as follows:

\begin{defn}
Define the space
$$
\Stab_\mass(\cD)\coloneqq\Stab(\cD)/\sim_m, \quad \text{ where } \sigma_1\sim_m\sigma_2 \text{ if } d_\mass(\sigma_1,\sigma_2)=0,
$$
It admits a natural generalized metric given by
$$
d_\mass(\sigma_1,\sigma_2) = \sup_{E\in D\backslash\{0\}}\left\{\left|\log\frac{m_{\sigma_1}(E)}{m_{\sigma_2}(E)}\right|\right\}.
$$
\end{defn}

\begin{rmk}
\label{rmk:R-action-on-Mass}
Recall that the free $\bC$-action on the stability space is given by
$$
\sigma=(Z,P)\mapsto\sigma.(x,y)\coloneqq\left(e^{x+i\pi y}\cdot Z, P'\right),  \text{ where }
(x,y)\in\bR^2\cong\bC \text{ and }
P'(\phi)=P(\phi-y).
$$
Observe that the $y$-component preserves masses and thus acts trivially on $\Stab_\mass(\mathcal{D})$,
while the $x$-component induces a free $\mathbb{R}$-action on $\Stab_\mass(\mathcal{D})$.
The quotient $\Stab_\mass(\mathcal{D})/\mathbb{R}$ inherits a natural metric:
$$
\bar{d}_\mass(\bar{\sigma}_1,\bar{\sigma}_2) =
    \inf_{x\in\bR}\left\{
    \sup_{E\in D\backslash\{0\}}\left\{\left|\log\frac{m_{\sigma_1}(E)}{m_{\sigma_2}(E)}-x\right|\right\}\right\}.
$$
\end{rmk}

\begin{eg}
Consider the space $\Stab_\mass(\Db(\bP^1))$. Proposition~\ref{prop:mass-projection} shows that the equivalence relation $\sim_m$ identifies algebraic stability conditions with boundary geometric ones, yielding
$\Stab_\mass(\Db(\bP^1))\cong\overline{\Stab^\Geo_\mass(\Db(\bP^1))}$. 
Furthermore, Proposition~\ref{prop:geometric-distance/C} gives an isometric identification of the quotient space:
$$
\left(\Stab_\mass(\Db(\mathbb{P}^1))/\mathbb{R},\bar{d}_\mass\right)\cong\left(\bH\cup(\bR\backslash\bZ),\frac{1}{2}d_\bZ\right),
$$
where $d_\bZ$ is the metric from Definition~\ref{defn:H-metric-Z}.
\end{eg}

We now summarize some metric properties of the space $\left(\mathbb{H} \cup (\mathbb{R}\backslash\mathbb{Z}), d_\mathbb{Z}\right)$, which consequently hold for the stability condition space $\left(\Stab_{\mass}(\Db(\mathbb{P}^1))/\mathbb{R}, \bar{d}_{\mass}\right)$:

\begin{prop}
\label{prop:metric-dZ-geodesic-etc}
The metric space $(\bH \cup (\bR \backslash \bZ), d_\bZ)$ satisfies:
\begin{enumerate}[label=(\alph*)]
    \item\label{item:geodesic} The space is geodesic (and therefore a length space). In fact, all $d_\hyp$-geodesics remaining geodesic under $d_\bZ$.
    \item\label{item:not-unique-geodesic} The space is not uniquely geodesic (and therefore not CAT(0)).
\end{enumerate}
\end{prop}

\begin{proof}[Proof of \ref{item:geodesic}]
Let $\tau_1,\tau_2\in\bH\cup(\bR\backslash\bZ)$. Choose a hyperbolic isometry $\rho\in\PSL(2,\bR)$ mapping these points to the imaginary axis with $\text{Im}(\rho(\tau_1)) > \text{Im}(\rho(\tau_2))$.
Since cross-ratios are preserved under such transformations, we have
$$
d_\bZ(\tau_1,\tau_2)=\sup_{s_1,s_2\in\bZ}\left(\log\frac{|\tau_1-s_2||\tau_2-s_1|}{|\tau_2-s_2||\tau_1-s_1|}\right)=\sup_{s_1,s_2\in\bZ}\left(\log\frac{|\rho(\tau_1)-\rho(s_2)||\rho(\tau_2)-\rho(s_1)|}{|\rho(\tau_2)-\rho(s_2)||\rho(\tau_1)-\rho(s_1)|}\right)
$$

\begin{center}
\begin{tikzpicture}[thick]
    \draw
      (-4,0) edge[-latex] node[at end,right]{} (4,0)
      (0,0) edge[-latex] node[at end,right]{} (0,4)
      (0,1)  node[scale=3](i){.} node[left]{$\rho(\tau_2)$}
      (0,2.5)  node[scale=3](i){.} node[left]{$\rho(\tau_1)$}
      (-2,0)  node[scale=1](i){x} node[below]{}
      (0.8,0)  node[scale=1](i){x} node[below]{$\rho(s)$}
    ;
\end{tikzpicture}
\end{center}

\noindent Since $\rho(\bZ)\subseteq\bR\cup\{\infty\}$, we have 
$$
d_\bZ(\tau_1,\tau_2)=\log\frac{|\rho(\tau_1)-m||\rho(\tau_2)-M|}{|\rho(\tau_2)-m||\rho(\tau_1)-M|},
$$
where
$$
m\coloneqq\inf_{s\in\bZ}|\rho(s)| \quad \text{ and } \quad M\coloneqq\sup_{s\in\bZ}|\rho(s)|.
$$

To show that hyperbolic geodesics remain geodesics under $d_\bZ$, by Lemma~\ref{lemma:geodesiccriterion}, it suffices to verify the additivity condition
$$
d_\bZ(\tau_1,\tau_2)=d_\bZ(\tau_1,\tau_3)+d_\bZ(\tau_3,\tau_2)
$$
for any $\tau_3$ lying on the hyperbolic geodesic between $\tau_1$ and $\tau_2$.
Under the isometry $\rho$, the intermediate point $\tau_3$ maps to a point on the imaginary axis between $\rho(\tau_1)$ and $\rho(\tau_2)$.
Observe that
$$
d_\bZ(\tau_1,\tau_3)=\log\frac{|\rho(\tau_1)-m||\rho(\tau_3)-M|}{|\rho(\tau_3)-m||\rho(\tau_1)-M|}
\quad \text{ and } \quad
d_\bZ(\tau_3,\tau_2)=\log\frac{|\rho(\tau_3)-m||\rho(\tau_2)-M|}{|\rho(\tau_2)-m||\rho(\tau_3)-M|}.
$$
Thus, we obtained the desired equality $d_\bZ(\tau_1,\tau_2)=d_\bZ(\tau_1,\tau_3)+d_\bZ(\tau_3,\tau_2)$.
\end{proof}

\begin{proof}[Proof of \ref{item:not-unique-geodesic}]
Consider the hyperbolic isometry $\rho(z)\coloneqq10-\frac{1}{z}$. Then
$$
\{9,\infty\} \subseteq \rho(\bZ) \subseteq [9,11] \cup \{\infty\}.
$$
\begin{center}
\begin{tikzpicture}[thick,scale=0.6]
    \draw
      (-8,0) edge[-latex] node[at end,right]{} (8,0)
      (0,0) edge[-latex] node[at end,right]{} (0,8)
      (0,5)  node[scale=3](i){.} node[left]{$9i$}
      (0,7)  node[scale=3](i){.} node[left]{$11i$}
      (5,0)  node[scale=1](i){x} node[below]{$9$}
      (7,0)  node[scale=1](i){x} node[below]{$11$}
      (6.5,0)  node[scale=1](i){x} node[below]{}
      (5.5,0)  node[scale=1](i){x} node[below]{}
      (5.75,0)  node[scale=1](i){x} node[below]{}
      (6.25,0)  node[scale=1](i){x} node[below]{}
      (6.125,0)  node[scale=1](i){x} node[below]{}
      (5.875,0)  node[scale=1](i){x} node[below]{}
      (6.0625,0)  node[scale=1](i){x} node[below]{}
      (5.9375,0)  node[scale=1](i){x} node[below]{}
      (6,0)  node[scale=1](i){x} node[below]{$\rho(\bZ)$}
      (0.2,6) node[scale=3](i){} node[right]{$\epsilon+10i$}
    ;
    \draw (0,5) -- (0.2,6);
    \draw (0,7) -- (0.2,6);
\end{tikzpicture}
\end{center}

Observe that for sufficiently small $\epsilon>0$, any two points $x,y$ on the segments
$$
[9i,\epsilon+10i]\cup[\epsilon+10i,11i]
$$
with $\text{Im}(x)>\text{Im(y)}$ satisfy the following:
\begin{itemize}
    \item The supremum
$$
\sup_{t\in[9,11]\cup\{\infty\}}\frac{|x-t|}{|y-t|}
$$
is attained at $t=9$.
    \item The infimum
$$
\inf_{t\in[9,11]\cup\{\infty\}}\frac{|x-t|}{|y-t|}
$$
is attained at $t=\infty$.
\end{itemize}
By the same reasoning as in the proof of \ref{item:geodesic}, the union
$$
\rho^{-1}([9i,\epsilon+10i]) \cup \rho^{-1}([\epsilon+10i,11i])
$$
forms a geodesic connecting $\rho^{-1}(9i)$ and $\rho^{-1}(11i)$ in $(\bH\cup(\bR\backslash\bZ),d_\bZ)$ for sufficiently small $\epsilon>0$.
This shows that the metric space is not uniquely geodesic.
\end{proof}

\subsection{Length metric.}
The second metric we consider is the general \emph{length metric} construction:

\begin{defn}
\label{defn:length-metric}
Let $(X,d)$ be a metric space. The \emph{length metric} associated with $d$ is the function $d_\ell\colon X\times X\rightarrow[0,\infty]$ defined by
$$
d_\ell(x,y)\coloneqq \inf  L(\gamma)
$$
where the infimum is taken over all curves $\gamma\colon[0,1]\rightarrow X$ with $\gamma(0)=x$ and $\gamma(1)=y$.
\end{defn}

Note that $(X,d_\ell)$ is automatically a length space.
Moreover, since $d\leq d_\ell$, the topology induced by $d_\ell$ is always at least as fine as that induced by $d$. However, this topology can be strictly finer in general.
In this context, we establish the following result:

\begin{prop}
Assume that $(\overline{\Stab^\Geo(\Db(\bP^1))}/\bC,\bar{d})$ is a length space. Then the topology of $(\Stab(\Db(\bP^1))/\bC,\bar{d})$ coincides with the topology induced by its associated length metric $\bar{d}_\ell$. Moreover, for all $\bar{\sigma}_1,\bar{\sigma}_2\in\Stab(\Db(\bP^1))/\bC$, the following inequalities hold:
$$
\bar{d}(\bar{\sigma}_1,\bar{\sigma}_2)\leq \bar{d}_\ell(\bar{\sigma}_1,\bar{\sigma}_2)\leq 2\bar{d}(\bar{\sigma}_1,\bar{\sigma}_2).
$$
\end{prop}

\begin{proof}
By assumption, the equality $\bar{d}(\bar{\sigma}_1,\bar{\sigma}_2)=\bar{d}_\ell(\bar{\sigma}_1,\bar{\sigma}_2)$ holds for all $\bar{\sigma}_1,\bar{\sigma}_2\in\Stab^\Geo(\Db(\bP^1))/\bC$.
The remaining cases (up to swapping $\bar{\sigma}_1$ and $\bar{\sigma}_2$) are:
\begin{enumerate}[label=(\alph*)]
    \item $\bar{\sigma}_1\in\Stab^\Geo(\Db(\bP^1))/\bC$ and $\bar{\sigma}_2\in (X_k\backslash\Stab^\Geo(\Db(\bP^1)))/\bC$ for some $k\in\bZ$,
    \item $\bar{\sigma}_1, \bar{\sigma}_2\in (X_k\backslash\Stab^\Geo(\Db(\bP^1)))/\bC$ for some $k\in\bZ$,
    \item $\bar{\sigma}_1\in (X_k\backslash\Stab^\Geo(\Db(\bP^1)))/\bC$ and $\bar{\sigma}_2\in (X_\ell\backslash\Stab^\Geo(\Db(\bP^1)))/\bC$ for some $k\neq\ell$.
\end{enumerate}

\noindent\textbf{Case (a):} Following the notation in Section~\ref{subsec:geometry-algebraic-distance}, let $\bar{\sigma}_1=\bar{\sigma}_{\alpha_1,\beta_1}$ and $\bar{\sigma}_2=\bar{\sigma}_{\alpha_2,\beta_2}$, where $\alpha_1,\alpha_2\in\bR$ and $0<\beta_1<1\leq\beta_2$.
By Proposition~\ref{prop:geometric-algebraic-distance/C}, we have:
$$
\bar{d}(\bar{\sigma}_{\alpha_1,\beta_1},\bar{\sigma}_{\alpha_2,\beta_2})=\frac{1}{2}\max\left\{d_\bZ\left(k+\frac{1}{1-e^{\alpha_1+i\pi\beta_1}},k+\frac{1}{1+e^{\alpha_2}}\right),\beta_2-\beta_1\right\}.
$$
By assumption, for any $\epsilon>0$, there exists a path $\gamma_1$ in $\overline{\Stab^\Geo(\Db(\bP^1))}/\bC$ connecting $\bar{\sigma}_{\alpha_1,\beta_1}$ to $\bar{\sigma}_{\alpha_2,1}$ satisfying:
$$
\bar{d}(\bar{\sigma}_{\alpha_1,\beta_1},\bar{\sigma}_{\alpha_2,1})\leq L(\gamma_1)<\bar{d}(\bar{\sigma}_{\alpha_1,\beta_1},\bar{\sigma}_{\alpha_2,1})+\epsilon.
$$
Let $\gamma_2$ be the straight path from $\bar{\sigma}_{\alpha_2,1}$ to $\bar{\sigma}_{\alpha_2,\beta_2}$ without varying $\alpha_2$. Combining these paths gives:
\begin{align*}
\bar{d}_\ell(\bar{\sigma}_{\alpha_1,\beta_1},\bar{\sigma}_{\alpha_2,\beta_2}) & \leq L(\gamma_1) + L(\gamma_2) \\
& \leq \bar{d}(\bar{\sigma}_{\alpha_1,\beta_1},\bar{\sigma}_{\alpha_2,1})+\epsilon + \frac{1}{2}(\beta_2-1) \\
& = \frac{1}{2}\max\left\{d_\bZ\left(k+\frac{1}{1-e^{\alpha_1+i\pi\beta_1}},k+\frac{1}{1+e^{\alpha_2}}\right),1-\beta_1\right\} + \frac{1}{2}(\beta_2-1) + \epsilon \\
& \leq 2\bar{d}(\bar{\sigma}_{\alpha_1,\beta_1},\bar{\sigma}_{\alpha_2,\beta_2})+\epsilon.
\end{align*}

\noindent\textbf{Case (b):} 
Since $(X_k\backslash\Stab^\Geo(\Db(\bP^1)))/\bC$ forms a geodesic space by \cite{Kikuta}*{Section~3.2}, the equality 
$\bar{d}(\bar{\sigma}_1,\bar{\sigma}_2)=\bar{d}_\ell(\bar{\sigma}_1,\bar{\sigma}_2)$ 
holds for all 
$\bar{\sigma}_1, \bar{\sigma}_2\in (X_k\backslash\Stab^\Geo(\Db(\bP^1)))/\bC$. \\

\noindent\textbf{Case (c):}
Let $\bar{\sigma}_1=\bar{\sigma}_{\alpha_1,\beta_1}\in X_k$ and $\bar{\sigma}_2=\bar{\sigma}_{\alpha_2,\beta_2}\in X_\ell$. By Proposition~\ref{prop:alg-stab-dist}, we have:
$$
\bar{d}(\bar{\sigma}_{\alpha_1,\beta_1},\bar{\sigma}_{\alpha_2,\beta_2})=\frac{1}{2}\max\left\{d_\bZ\left(k+\frac{1}{1+e^{\alpha_1}},\ell+\frac{1}{1+e^{\alpha_2}}\right),\beta_1+\beta_2-1
\right\}
$$
By assumption, for any $\epsilon>0$, there exists a path $\gamma_1$ in $\overline{\Stab^\Geo(\Db(\bP^1))}/\bC$ connecting $\bar{\sigma}_{\alpha_1,1}$ to $\bar{\sigma}_{\alpha_2,1}$ satisfying: 
$$
\bar{d}(\bar{\sigma}_{\alpha_1,1},\bar{\sigma}_{\alpha_2,1})\leq L(\gamma_1)<\bar{d}(\bar{\sigma}_{\alpha_1,1},\bar{\sigma}_{\alpha_2,1})+\epsilon.
$$
Let $\gamma_2$ be the straight path from $\bar{\sigma}_{\alpha_1,1}$ to $\bar{\sigma}_{\alpha_1,\beta_1}$, and  $\gamma_3$ be the straight path from $\bar{\sigma}_{\alpha_2,1}$ to $\bar{\sigma}_{\alpha_2,\beta_2}$. Combining these paths gives:
\begin{align*}
\bar{d}_\ell(\bar{\sigma}_{\alpha_1,\beta_1},\bar{\sigma}_{\alpha_2,\beta_2}) & \leq L(\gamma_1) + L(\gamma_2) + L(\gamma_3) \\
& \leq \bar{d}(\bar{\sigma}_{\alpha_1,1},\bar{\sigma}_{\alpha_2,1})+\epsilon  +\frac{1}{2}(\beta_1-1)+\frac{1}{2}(\beta_2-1) \\
& = \frac{1}{2}\max\left\{d_\bZ\left(k+\frac{1}{1+e^{\alpha_1}},\ell+\frac{1}{1+e^{\alpha_2}}\right),1\right\} +\frac{1}{2}(\beta_1-1) + \frac{1}{2}(\beta_2-1) + \epsilon \\
& \leq 2\bar{d}(\bar{\sigma}_{\alpha_1,\beta_1},\bar{\sigma}_{\alpha_2,\beta_2})+\epsilon.
\end{align*}
\end{proof}

\bigskip
\bibliography{ref}
\bibliographystyle{alpha}

\ \\

\noindent Yu-Wei Fan \\
\textsc{Yau Mathematical Sciences Center, Tsinghua University\\
Beijing 100084, China}\\
\texttt{ywfan@mail.tsinghua.edu.cn}

\end{document}